\documentclass[12pt]{amsart}

\usepackage{a4wide,amsmath,amssymb}
\usepackage[toc,page]{appendix}
\usepackage{hyperref}

\usepackage{url}

\usepackage{mathtools}

\usepackage{xcolor}

\newtheorem{thm}{Theorem}[section]
\newtheorem{lemma}[thm]{Lemma}

\newtheorem{cor}[thm]{Corollary}
\theoremstyle{remark}
\newtheorem{rem}[thm]{Remark}

\newtheorem{ex}[thm]{Example}

\theoremstyle{definition}

\newtheoremstyle{Claim}{}{}{\itshape}{}{\itshape\bfseries}{:}{ }{#1}
\theoremstyle{Claim}

\newcommand{\R}{\mathbb{R}}

\newcommand{\eps}{\varepsilon}

\makeatletter
\@namedef{subjclassname@2020}{\textup{2020} Mathematics Subject Classification}
\makeatother

\theoremstyle{plain}

\begin{document}

\title[]{Interior H\"older and Calder\'on-Zygmund estimates for fully nonlinear equations with natural gradient growth}

\author{Alessandro Goffi}
\address{Dipartimento di Matematica e Informatica ``Ulisse Dini'', Universit\`a degli Studi di Firenze, 
viale G. Morgagni 67/A, 50134 Firenze (Italy)}
\curraddr{}
\email{alessandro.goffi@unifi.it}

\subjclass[]{}

 \thanks{
The author wishes to thank Luca Rossi and Andrzej {\'S}wi{\c{e}}ch for some comments on the first draft of the manuscript and for pointers to references. The author is member of the Gruppo Nazionale per l'Analisi Matematica, la Probabilit\`a e le loro Applicazioni (GNAMPA) of the Istituto Nazionale di Alta Matematica (INdAM). The author was partially supported by the INdAM-GNAMPA Projects 2023 and 2024, by the King Abdullah University of Science and Technology (KAUST) project CRG2021-4674 ``Mean-Field Games: models, theory and computational aspects" and by the project funded by the EuropeanUnion – NextGenerationEU under the National Recovery and Resilience Plan (NRRP), Mission 4 Component 2 Investment 1.1 - Call PRIN 2022 No. 104 of February 2, 2022 of Italian Ministry of University and Research; Project 2022W58BJ5 (subject area: PE - Physical Sciences and Engineering) ``PDEs and optimal control methods in mean field games, population dynamics and multi-agent models" .}

\date{\today}



\date{\today}

\begin{abstract}
We establish local H\"older estimates for viscosity solutions of fully nonlinear second order equations with quadratic growth in the gradient and unbounded right-hand side in $L^q$ spaces, for an integrability threshold $q$ guaranteeing the validity of the maximum principle. This is done through a nonlinear Harnack inequality for nonhomogeneous equations driven by a uniformly elliptic Isaacs operator and perturbed by a Hamiltonian term with natural growth in the gradient. As a byproduct, we derive a new Liouville property for entire $L^p$ viscosity solutions of fully nonlinear equations as well as a nonlinear Calder\'on-Zygmund estimate for strong solutions of such equations. 
\end{abstract}

\maketitle

\section{Introduction}

The aim of this note is the study of regularity properties for suitable solutions of the fully nonlinear elliptic equation
\begin{equation}\label{HJell}
F(x,D^2u)+H(x,u,Du)=0\text{ in }\Omega\subset\R^n,
\end{equation}
where $F$ is measurable in $x$ and uniformly elliptic with given ellipticity constants $0<\lambda\leq\Lambda$, i.e.
\[
\mathcal{M}_{\lambda,\Lambda}^-(M-N)\leq F(x,M)-F(x,N)\leq \mathcal{M}_{\lambda,\Lambda}^+(M-N),\ M,N\in\mathcal{S}_n,
\]
$\mathcal{M}_{\lambda,\Lambda}^\pm$ being the Pucci's extremal operators, and $H$ satisfies the following natural growth condition
\begin{equation}\label{assH}
|H(x,u,Du)|\leq C_{H}|Du|^2+f(x).
\end{equation}
Here, $f\in L^q$ with integrability $q>q_{E}$, $q_E=q_E(n,\Lambda/\lambda)\in[n/2,n)$ is the so-called Escauriaza constant, which determines the range in which the maximum principle holds and depends upon the so-called ellipticity $\Lambda/\lambda$, cf. \cite{CrandallSwiech,Escauriaza,FabesStroock}. The Pucci's extremal operators are defined respectively as $\mathcal{M}_{\lambda,\Lambda}^+(M)=\sup\{\mathrm{Tr}(AM):\ A\in\mathcal{S}_n,\ \lambda I_n\leq A\leq \Lambda I_n\}$ and  $\mathcal{M}_{\lambda,\Lambda}^-(M)=\inf\{\mathrm{Tr}(AM):\ A\in\mathcal{S}_n,\ \lambda I_n\leq A\leq \Lambda I_n\}$, where $\mathcal{S}_n$ is the set of $n\times n$ symmetric matrices. \\
The first step of our program is a low-regularity estimate in $C^\alpha$ spaces, $\alpha\in(0,1)$: under the above assumptions, we first prove that $L^p$ viscosity solutions of \eqref{HJell} are $C^\alpha_{\mathrm{loc}}$, i.e. they satisfy
\[
|u(x)-u(y)|\leq C|x-y|^\alpha,\ x,y\in B_\frac{R}{2},
\]
with $C$ depending on $n,q,\lambda,\Lambda,\|u\|_{L^\infty(B_R)},C_H$, via a nonlinear (invariant) Harnack inequality, see Section \ref{sec;holder}. This in turn amounts to prove that any positive and bounded $L^p$ viscosity solution $u$ to \eqref{HJell} satisfies
\[
\sup_{B_{\frac{R}{2}}}u\leq C\left(\inf_{B_\frac{R}{2}}u+R^{2-\frac{n}{q}}\|f\|_{L^q(B_R)}\right),
\]  
where $C$ is a positive constant depending on $n,q,\lambda,\Lambda,\|u\|_{L^\infty(B_R)},C_H$, cf. Theorem \ref{Harnack}, but not on $R$. The proof of the Harnack inequality is inspired from the corresponding counterpart available for divergence-type equations, cf. \cite{BensoussanFrehse,HanLin}, which in turn exploits the De Giorgi-Nash-Moser theory. The keystones of the proof of the $C^\alpha$ estimate, which has a linear flavor, are a weak Harnack inequality for supersolutions and a local maximum principle for subsolutions applied to suitable equations driven by a Pucci's extremal operator and with lower-order coefficients, cf. Section \ref{sec;prel}. This is done via exponential transformations applied to the viscosity formulation of \eqref{HJell} that can be regarded as fully nonlinear versions of the classical Hopf-Cole change of variable applied in the theory of viscous Hamilton-Jacobi equations. Notably, the invariant Harnack inequality implies also a new Liouville property for entire $L^p$ viscosity solutions of the following homogeneous equation
\[
F(x,D^2u)+H(Du)=0\text{ in }\R^n,
\]
where $F$ is any uniformly elliptic operator (e.g. in Isaacs form) and $H$ has quadratic growth, see Corollary \ref{liouville}. Recall that the classical Hopf-Cole transform does not apply directly to \eqref{HJell} for two reasons: $F$ is nonlinear in $D^2u$ and $H$ is not purely quadratic. Remarkably, when $F$ is linear we find a new Liouville property for second order Hamilton-Jacobi equations under a weaker notion of solution than the earlier results, cf. \cite{PeletierSerrin,Lions85}. When $H$ has general superlinear power-growth, nonexistence results can be derived from Bernstein-type estimates obtained by the doubling of variables technique \cite{CDLP,BDL,FuentesQuaas}, though they are restricted either to special fully nonlinear second-order operators 1-homogeneous in $D^2u$ or to linear nondivergence diffusions. In the special case of the quadratic gradient growth our approach provides a Liouville result without any other assumption on $F$ in $D^2u$ other than the uniform ellipticity.\\
Then, under the assumptions guaranteeing Calder\'on-Zygmund estimates for the controlled diffusive equation $G(x,D^2v)=g(x)$, see \cite{C91,CC,Escauriaza,CCKS}, we prove the following interior maximal regularity estimate for strong solutions to \eqref{HJell} that reads as
\[
\|D^2u\|_{L^q(B_\frac12)}+\||Du|^2\|_{L^q(B_\frac12)}\leq C.
\]
Hessian estimates in $L^q$ for $v$ solving $G(x,D^2v)=g(x)$ can be proved in many interesting cases that require either concavity type assumptions on $G$ in the Hessian variable, cf. e.g. \cite{Gjmpa} and the references therein, or $G$ to be an Isaacs operator with a special structure, see \cite{CCjmpa, Kovats} and Remark \ref{Isaacs}. This step is achieved by a perturbation argument via the Miranda-Nirenberg interpolation inequality \cite{Miranda,Nirenbergint}. Note that when $\Lambda/\lambda\to1$, $F(x,D^2u)=\Delta u$ and $q_E=n/2$, cf. \cite[p. 420]{Escauriaza} and the interior maximal regularity estimates are new. In particular, in the case of linear diffusions, we shed new light on a conjecture raised by P.-L. Lions about the regularity of viscous Hamilton-Jacobi equations with unbounded terms and diffusions in nondivergence form, see Section \ref{sec;conj} for a thorough discussion. We mention that the recent paper \cite{Teixeira} addresses high regularity properties in H\"older spaces for equations similar to \eqref{HJell}.\\
The methods of the paper apply also to the parabolic equation
\begin{equation}\label{HJpar}
F(x,t,D^2u)-\partial_t u+H(x,t,u,Du)=0\text{ in }Q_1:=B_1\times[-1,0),
\end{equation}
with $H$ satisfying
\begin{equation}\label{assHpar}
|H(x,t,u,Du)|\leq C_{H}|Du|^2+f(x,t),
\end{equation}
where $f\in L^q_{x,t}$, $q>q_E\in[\frac{n+2}{2},n+1)$, and lead to the following maximal regularity estimate
\[
\|\partial_t u\|_{L^q(Q_\frac12)}+\|D^2u\|_{L^q(Q_\frac12)}+\||Du|^2\|_{L^q(Q_\frac12)}\leq C,
\]
which is valid above the (parabolic) Escauriaza exponent. It is enough to exploit suitable weak Harnack inequalities and generalized maximum principles on parabolic cylinders. \\
Remarkably, the results apply in particular to nondivergence structure equations of the form
\[
\mathrm{Tr}(A(x)D^2u)+H(x,u,Du)=f(x)
\]
and to its time-dependent counterpart, where $f\in L^q$ and $H$ has natural gradient growth.They provide a $W^{2,q}_{\mathrm{loc}}$ estimate assuming that for any $q>q_E$ there exists $\eps>0$ depending on $n,\lambda,\Lambda,q$ such that
\[
\frac{1}{|B_r(x_0)|}\left(\int_{B_r(x_0)}|a_{ij}(x)-a_{ij}(x_0)|^n\,dx\right)^\frac1n\leq \eps,\ \text{ for any }B_r(x_0)\subset B_1.
\]
This requirement weakens the classical continuity estimate on $A$ leading to interior maximal $L^q$-regularity estimates, as it measures the distance from a constant coefficient equation in $L^n$ norm rather than $L^\infty$, cf. e.g. \cite{CC,GT}. The best result for this nondivergence class appeared recently in \cite{CV} for superquadratic problems, and it depends on the modulus of continuity of $A$ requiring also $A\in W^{1,n}$.
Some results in this direction for linear diffusions with discontinuous coefficients are already available when $q\geq n$ under VMO assumptions on $A$, see e.g. Section 2.6 in \cite{Maugeri}. \\
We recall that the study of a priori estimates for nonlinear equations with quadratic growth via interpolation methods is standard in the context of semilinear and quasilinear PDEs. It was initiated by H. Amann-M.G. Crandall \cite{AmannCrandall}, A. Bensoussan-J. Frehse \cite{BensoussanFrehse}, A. Maugeri-D. Palagachev-L. Softova \cite{Maugeri}, and also exploited in more recent works as \cite{CGpar,Gccm}.\\
The motivation of studying these quantitative issues, apart from their own interest in the regularity theory of nonlinear equations, stems from the theory of Mean Field Games introduced by J.-M. Lasry and P.-L. Lions \cite{LL}. In particular, they introduced in Section 2.7 of \cite{LL} the following general system of backward-forward equations arising from a drift-diffusion controlled dynamics (equipped with suitable boundary conditions that we do not display here)
\[
\begin{cases}
-\partial_t u-F(x,t,D^2u)+H(Du)=g(m)&\text{ in }\Omega\times(0,T),\\
\partial_t m-\sum_{i,j=1}^n\partial_{ij}(F_{ij}(x,t,D^2u)m)-\mathrm{div}(D_pH(Du)m)=0&\text{ in }\Omega\times(0,T),\\
u(x,T)=V[m],\ m(x,0)=m_0(x)&\text{ in }\Omega.
\end{cases}
\]
Here, $F=F(\cdot ,N)$ can be seen by extension in $\R^{n\times n}$ as a function of $n\times n$ variables and $F_{ij}(\cdot,N)=\frac{\partial F}{\partial n_{ij}}(\cdot,N)$.
It is well-known that maximal regularity for the first equation unlocks classical regularity for the whole system \cite{CGpar,GP} as a consequence of the couplings among the equations and the peculiar backward-forward structure. 

\section{Preliminaries: notions of solutions, local maximum principle and weak Harnack inequality}\label{sec;prel}
We recall some standard notions of solutions taken from \cite{CCKS}. These are needed since $F$ (the operator) and $f$ (the source) could not be continuous at points in their domains. We say that $u\in C(\Omega)$ is a $L^p$ viscosity subsolution (resp. supersolution) to $F(x,u,Du,D^2u)=f$ in $\Omega$ with $F$ proper and measurable, $p>n/2$, $f\in L^p(\Omega)$, if for all $\varphi\in W^{2,p}_{\mathrm{loc}}(\Omega)$, and point $\bar x\in\Omega$ at which $u-\varphi$ attains a local maximum (minimum), one has
\[
\mathrm{ess\ lim\ inf}_{x\to\bar x}(F(x,u(x),D\varphi(x),D^2\varphi(x))-f(x))\leq 0
\]
\[
(\mathrm{ess\ lim\ sup}_{x\to\bar x}(F(x,u(x),D\varphi(x),D^2\varphi(x))-f(x))\geq 0).
\]
Then, $u$ is a $L^p$ viscosity solution if it is both a viscosity sub- and supersolution. Clearly, if $\varphi\in C^2$ the definition reduces to the classical notion of viscosity solution.\\
Finally, we will use the notion of strong $W^{2,p}_{\mathrm{loc}}(\Omega)$ solution to $F=f$ in the usual sense, i.e. when the equation holds a.e. in $\Omega$. \\

We recall now the following ``half-Harnack'' inequalities for semisolutions of fully nonlinear equations. From now on, when considering $L^p$ viscosity solutions, we assume $p>q_E$. The results are classical and are the matter of \cite{C91,CC}, while we present here some refinements due to \cite[Theorem A]{SirakovSouplet}, \cite{KoikeSwiechJapan} and \cite{AmendolaRossiVitolo}. We denote by $q_E\in[n/2,n)$ the Escauriaza exponent \cite{Escauriaza}. In what follows $\mathcal{M}^\pm_{\lambda,\Lambda}$ denote the Pucci's extremal operators.
\begin{thm}
Suppose $g,h\in L^q(\Omega)$, $\Omega\subseteq\R^n$, $q>q_E$. We have the following
\begin{itemize}
\item (weak Harnack inequality) There exist constants $\eps$ and $C_W$ depending only on $q,n,\lambda,\Lambda$ such that if $v\geq0$ is a $L^p$ viscosity solution of $\mathcal{M}^-_{\lambda,\Lambda}(D^2v)-gv\leq g$ in $\Omega$, then 
\[
\left(\frac{1}{R^n}\int_{B_R}v^\eps\,dx\right)^\frac1\eps\leq C_W\left(\inf_{B_{\frac{R}{2}}}v+R^{2-\frac{n}{q}}\|g\|_{L^q(B_{R/2})}\right),
\]
where $C_W$ depends only on $n,q,\lambda,\Lambda$.
\item (Local maximum principle) Suppose $\mathcal{M}^+_{\lambda,\Lambda}(D^2v)\geq h+hv$ in $\Omega$ in the $L^p$-viscosity sense. Then, for any $s>0$ there exists a constant $C_M>0$ depending on $n,q,\lambda,\Lambda,s$ such that
\[
\sup_{B_{\frac{R}{2}}}v\leq C_M\left(\frac{1}{R^\frac{n}{s}}\|v^+\|_{L^s(B_R)}+R^{2-\frac{n}{q}}\|h\|_{L^q(B_R)}\right).
\]
\end{itemize}
\end{thm}
Similar results hold for the fully nonlinear evolution equation
\begin{equation}
G(x,t,D^2v)-\partial_t v=g(x,t).
\end{equation}
We refer to \cite{Wang1,Wang2} for more details.

\section{Harnack inequality for equations with quadratic growth and applications to Liouville properties}\label{sec;holder}
We now present the following important nonlinear Harnack inequality for Hamilton-Jacobi-Bellman and -Isaacs equations of the form \eqref{HJell}. It slightly improves Theorem 6.9 in \cite{KoikeBook} and Theorem 3.5 in \cite{AmendolaRossiVitolo} with respect to the integrability requirements on the right-hand side, though the proof is similar and based on exponential transformations. This approach dates back to \cite{TrudingerRevista}, see Theorem 5.1 therein, where a bounded right-hand side was considered and the results proved for continuous viscosity solutions.
\begin{thm}\label{Harnack}
Let $u$ be a nonnegative and bounded $L^p$ viscosity solution of \eqref{HJell} satisfying \eqref{assH}. Then it satisfies
\[
\sup_{B_{\frac{1}{2}}}u\leq \overline{C}e^{2\frac{C_H}{\lambda}\|u\|_\infty}\left(\inf_{B_{\frac{1}{2}}}u+\|f\|_{L^q(B_1)}\right),
\]  
where $\overline{C}$ is a positive constant depending on $n,q,\lambda,\Lambda$. Moreover, when $f=0$ the following scaled version of the Harnack inequality holds
\[
\sup_{B_{\frac{R}{2}}}u\leq C\inf_{B_\frac{R}{2}}u,
\]
where $C$ is a positive constant depending on $n,q,\lambda,\Lambda,\|u\|_\infty,C_H$, but not on $R$.
\end{thm}
\begin{proof}
The proof follows the argument of Theorem 6.9 in \cite{KoikeBook}, to which we refer for a detailed proof in the case $q\geq n$. Since the steps are the same, we provide only the main ideas. It is enough to note that $v=\frac{1}{\beta}(e^{\beta u}-1)$ for $\beta=C_H/\lambda$ is a $L^p$ viscosity subsolution to an equation with a zero-th order coefficient and a right-hand side in $L^q$ spaces driven by $\mathcal{M}^+_{\lambda,\Lambda}$ in view of Lemma 2.3 in \cite{Sirakov}, see also \cite{SwiechDCDS} or \cite{KoikeBook}. Indeed, if $u$ solves \eqref{HJell} in the $L^p$ viscosity sense, we have
\[
\mathcal{M}^+_{\lambda,\Lambda}(D^2u)+C_H|Du|^2\geq-f(x)
\]
in the $L^p$ viscosity sense. Then, $v=\frac{1}{\beta}(e^{\beta u}-1)$ for $\beta=C_H/\lambda$ solves
\[
\frac{\mathcal{M}^+_{\lambda,\Lambda}(D^2v)}{1+\beta v}\geq \mathcal{M}^+_{\lambda,\Lambda}(D^2u)+C_H|Du|^2.
\]
This implies in turn
\[
\mathcal{M}^+_{\lambda,\Lambda}(D^2v)\geq -(1+\beta v)f(x).
\]
 We can apply the local maximum principle to conclude for all $s>0$
\begin{equation}\label{LMP}
\sup_{B_{\frac{1}{2}}}u\leq \sup_{B_{\frac{1}{2}}}v\leq C_M\left(\|v\|_{L^s(B_1)}+\|f\|_{L^q(B_1)}\right).
\end{equation}
On the other hand, we have that $w=\frac{1}{\beta}(1-e^{-\beta u})$ is a $L^p$ viscosity supersolution to an equation driven by $\mathcal{M}^-_{\lambda,\Lambda}$ again by Lemma 2.3 in \cite{Sirakov}. Indeed,
we have that 
\[
\mathcal{M}^-_{\lambda,\Lambda}(D^2u)-C_H|Du|^2\leq f(x).
\]
This implies that
\[
\frac{\mathcal{M}^-_{\lambda,\Lambda}(D^2w)}{1-\beta w}\leq f(x). 
\]
Then the weak Harnack inequality gives
\begin{equation}\label{weakHarnack}
\|w\|_{L^{\eps}(B_1)}\leq C_W\left(\inf_{B_{\frac{1}{2}}}w+\|f\|_{L^q(B_\frac12)}\right).
\end{equation}
We note now that $u\leq v\leq ue^{\beta \|u\|_\infty}$ and $ue^{-\beta\|u\|_\infty}\leq w\leq  u$. Therefore, combining these inequalities with \eqref{LMP} (applied with $s=\eps$) and \eqref{weakHarnack} we conclude for a suitable positive constant $\overline{C}$ the validity of the following chain of inequalities
\begin{align*}
\sup_{B_\frac12}u\leq \sup_{B_\frac12}v&\leq C_M\left(\|v\|_{L^s(B_1)}+\|f\|_{L^q(B_1)}\right)\leq C_M\left(e^{2\beta\|u\|_\infty}\|w\|_{L^s(B_1)}+\|f\|_{L^q(B_1)}\right)\\
&\leq C_M\left[C_W e^{2\beta\|u\|_\infty}\left(\inf_{B_{\frac{1}{2}}}w+\|f\|_{L^q(B_\frac12)}\right)+\|f\|_{L^q(B_1)}\right]\\
&\leq \overline{C}e^{2\frac{C_H}{\lambda}\|u\|_\infty}\left(\inf_{B_{\frac{1}{2}}}u+\|f\|_{L^q(B_1)}\right).
\end{align*}
\end{proof}
Classical methods for linear equations lead to the local H\"older continuity of viscosity solutions. H\"older bounds for fully nonlinear equations with quadratic growth conditions in the gradient appeared in \cite{Sirakov}, which do not exploit the Harnack inequality and are more flexible if one only needs a H\"older estimate. Some other developments can be found in \cite{KoikeSwiechJapan} for equations with linear gradient growth. The novelty of our H\"older estimate, whose proof is classical since it follows from the Harnack inequality (cf. Lemma 2 in \cite{Escauriaza} or \cite{KoikeBook}), is the requirement $q_E<q<n$ (in \cite{Sirakov} or \cite{KoikeBook}, instead, the proof requires $q=n$) and the treatment of equations with quadratic growth (in \cite{KoikeSwiechJapan} equations have at most linear gradient growth). More recent H\"older estimates for equations with subquadratic gradient growth are the matter of \cite{VitorNornberg}. Finally, the paper \cite{CDLP} contains a H\"older regularity estimate for subsolutions of fully nonlinear PDEs and superquadratic gradient terms (but with bounded right-hand side). The latter results are of different nature, as the regularity comes from the coercivity of the gradient term and not from the diffusion.
\begin{cor}\label{Holder}
Under the standing assumptions, we have that $L^p$ viscosity solutions to \eqref{HJell} are  $C^\alpha_{\mathrm{loc}}$ with $\alpha=\alpha(n,q,\lambda,\Lambda,C_H,\|u\|_\infty,\|f\|_{L^q})\in(0,1)$.
\end{cor}
\begin{proof}
The proof is standard and follows the argument of e.g. Corollary 4.18 in \cite{HanLin} or Theorem 6.10 in \cite{KoikeBook}.
\end{proof}
\begin{rem}
By an argument of L. Caffarelli \cite{C91} one can raise the regularity up to $C^{1,\alpha}$. Such regularity results under natural or subquadratic growth conditions can be found in \cite{Nornberg,SwiechDCDS,VitorNornberg}. These are the counterpart of Theorem 4.24 in \cite{HanLin} for equations in divergence form.
\end{rem}

\begin{rem}[Comparison with divergence structure equations]\label{div}
The previous argument to derive local H\"older estimates has been already used (and actually inspired from) a variational approach in the context of divergence-type equations
\[
\mathrm{div}(A(x)Du)=H(x,u,Du)
\]
in \cite{BensoussanFrehse} and \cite{HanLin} when $H$ satisfies \eqref{assH}, $f\in L^q$, $q>n/2$, and $A$ measurable and such that $A\in L^\infty$ satisfies $\lambda I_n\leq A\leq \Lambda I_n$. In this case, it is enough to observe that $v=\frac{1}{\alpha}(e^{\alpha u}-1)$ for a suitable large $\alpha$ is a weak $H^1$ subsolution to $\mathrm{div}(A(x)Dv)=g(x)$, $g\in L^q$, so that one can apply the local boundedness theorem of J. Moser, cf. Theorem 4.14 in \cite{HanLin}. Conversely, $w=\frac{1}{\alpha}(1-e^{-\alpha u})$ is a supersolution to a similar equation for large $\alpha$ and hence one can apply the weak Harnack inequality in Theorem 4.15 of \cite{HanLin} (here the bound on $\eps$ is explicit and $\eps<\frac{n}{n-2}$). The conclusion follows as in Theorem \ref{Harnack}. Higher regularity in $C^{1,\alpha}$ can be achieved under the additional assumptions that $A\in C^\alpha$, see e.g. \cite{HanLin}. \\
The procedure detailed in this section can be regarded as the counterpart of the results in \cite{HanLin} for nondivergence structure equations and the results appear to be new. Nonetheless, it has been partially used to derive weak Harnack inequalities or local maximum principles for equations with subquadratic growth in e.g. \cite{SwiechDCDS,AmendolaRossiVitolo}. As in the divergence setting, higher regularity estimates follow by assuming additional conditions on $F$, as discussed in the next section.
\end{rem}
Since the nonlinear Harnack inequality is invariant with respect to the radius of the ball $B_R$ we can derive a nonlinear Liouville property in the same way one obtains the classical Liouville theorem for bounded or semi-bounded harmonic functions. A review on Liouville properties for such equations can be found in \cite{CGproc}.
\begin{cor}\label{liouville}
If $u$ is a nonnegative and bounded $L^p$, $p>q_E$, viscosity solution to
\[
F(x,D^2u)+H(Du)=0\text{ in }\R^n,
\]
where $F$ is any uniformly elliptic operator and $H$ satisfies the natural growth conditions \eqref{assH} with $f=0$, then it must be a constant.
\end{cor}

\begin{rem}
The above result for equations with power-like gradient terms is usually deduced through (viscosity) Bernstein type gradient estimates: in the fully nonlinear case this is known for nondivergence operators, 1-homogeneous fully nonlinear operators or Pucci's extremal operators only, see e.g. respectively \cite{CDLP,BDL,FuentesQuaas}. We note that the Liouville property for rather general fully nonlinear diffusions can be deduced through the parabolic estimate in Theorem 4.19 of \cite{PorrettaPriola}, but it holds for continuous bounded solutions and subquadratic gradient terms. The case of linear diffusions in nondivergence form with Lipschitz coefficients was proved in \cite{PeletierSerrin}, see also \cite{Lions85}, via the more classical Bernstein method. In the case of diffusions driven by $\Delta u$ and a purely quadratic Hamiltonian $H(Du)=|Du|^2$, the Hopf-Cole transformation allows to obtain the Liouville property from the linear Harnack inequality assuming only a one-side bound on the solution. Such a transformation does not apply directly to problems with natural gradient growth. Moreover, this latter approach cannot be directly applied to fully nonlinear equations since the exponential change of variable leads to an inequality rather than an equality.  
\end{rem}
\section{Maximal $L^q$-regularity by interpolation and a (fully nonlinear) conjecture of P.-L. Lions}\label{sec;conj}
When Calder\'on-Zygmund estimates for the equation without lower-order terms
\[
G(x,D^2u)=g(x)
\]
are available, one can prove maximal regularity estimates in $L^q$ spaces for \eqref{HJell} by interpolation with the $C^\alpha$ bounds. The former estimates are true in many interesting cases, in particular when the constant coefficient equation has $C^{1,1}$ estimates: this is the case of concave/convex operators (see e.g. \cite{C91,Escauriaza}), but also when the level sets of the operator are convex, cf. \cite{CaffarelliYuan,Gjmpa}, or when $F$ is concave or convex at infinity \cite{CaffarelliYuan,HuangAIHP,Gjmpa}. The precise statement of the Calder\'on-Zygmund estimate is due to L. Caffarelli \cite{C91} and L. Escauriaza \cite{Escauriaza}, see also Appendix B of \cite{CCKS} and \cite{KS}.\\
Calder\'on-Zygmund estimates for \eqref{HJell} follow by interpolation via the Miranda-Nirenberg inequalities \cite{BensoussanFrehse,Miranda,Nirenbergint}, as it reduces the problem to a lower order estimate in $C^\alpha_{\mathrm{loc}}$. In particular, we emphasize that the results apply to the equation
\[
\mathrm{Tr}(A(x)D^2u)=H(x,u,Du),\  A\in C^0.
\]
We remark that all the known maximal regularity results for such diffusion operators require some differentiability properties of the diffusion coefficient (see e.g. \cite{BardiPerthame,CV,CGpar}).
We start by recalling a maximal $L^q$-regularity estimate of Calder\'on-Zygmund type for a fully nonlinear equation without lower-order terms: it needs the validity of $C^{1,1}$ estimate for the ``constant coefficient'' equation (i.e. for $F(x_0,D^2u)$ with $x_0$ fixed) along with a control on the oscillation of $F$ in $x$ with respect to the $L^n$ norm. The result below can be found in Theorem 1 of \cite{C91} and was refined in Theorem 1 of \cite{Escauriaza}.

\begin{lemma}\label{CZg}

Let $u\in W^{2,q}_{\mathrm{loc}}(\Omega)\cap L^q(\Omega)$ be a strong solution to $G(x,D^2u)=g(x)$, $g\in L^q(\Omega)$, $q>q_E$, $B_1\subset\Omega$, and $G$ uniformly elliptic with $G(x,0)=0$ and measurable in $x$. Suppose that the constant coefficient equation has $C^{1,1}$ interior estimates, namely, we suppose that $G(x_0,D^2 w)$ has interior $C^{1,1}$ estimates (with constant $c_e$) for any $x_0\in B_1$. We suppose also that, defined 
\[
\beta(x,x_0)=\sup_{M\in\mathcal{S}_n}\frac{|G(x,M)-G(x_0,M)|}{\|M\|+1},
\]
for any ball $B_R(x_0)\subset B_1$ it holds
\[
\left(|B_r(x_0)|^{-1}\int_{B_r(x_0)}\beta^n(x,x_0)\,dx\right)^\frac1n\leq\beta_0.
\]
Then for all $B_R\subset\subset\Omega$, $R<1$, and $\sigma\in(0,1)$ we have
\begin{equation}\label{ineqGT}
\|D^2u\|_{L^q(B_{\sigma R})}\leq \frac{C(n,q,\lambda,\Lambda,c_e,\beta_0)}{(1-\sigma)^2R^2}(R^2\|g\|_{L^q(B_R)}+\|u\|_{L^q(B_R)})\ ,\sigma\in(0,1).
\end{equation}
\end{lemma}
\begin{proof}
The proof follows the arguments of Theorem 9.11 in \cite{GT}, see in particular Appendix B and Lemma 3.1 of \cite{CCKS} or the arguments in Theorem 3.1 of \cite{KS}. When $G(x,D^2u)=\mathrm{Tr}(A(x)D^2u)$ the result can be found in eq. (9.40) of \cite{GT}, and the estimate depends on the modulus of continuity of $A$. The result in \cite[Theorem 1]{C91} does not depend on the modulus of continuity of $A$ when applied to linear operators. 
\end{proof}
\begin{rem}\label{weakconcave}
If $G(x_0,M)$ is concave or convex in $M$, the equation $F(x_0,D^2w)=0$ has $C^{1,1}$ estimates with a universal constant $c_e$: this is the first step towards the Evans-Krylov regularity \cite{CC}. Some other extensions under concavity conditions at infinity can be found in \cite{CaffarelliYuan,HuangAIHP}. The more recent work \cite{Gjmpa} discussed $C^{1,1}$ interior estimates when $F$ is quasiconcave and, in the parabolic regime, derived $C^{1,1}$ parabolic estimates when $F$ is quasiconcave, concave at infinity or close to a hyperplane. We also mention the analysis in Sobolev spaces $W^{2,q}$ carried out by N.V. Krylov on Bellman's equations with VMO coefficients, see \cite{KrylovVMO}.
\end{rem}
\begin{rem}\label{Isaacs}
The $W^{2,q}_{\mathrm{loc}}$ regularity for $G(x,D^2u)=g(x)$ holds also for some nonconvex/nonconcave equations in the Hessian. For instance, it holds when $G$ is the minimum, for fixed $x_0\in B_1$, between a concave and a convex operator \cite{CCjmpa}, or in the special case when 
\[
G(D^2u)=\Delta u+(u_{x_1x_1})^+-(u_{x_2x_2})_{-},
\]
see \cite{Kovats}. Finally, it holds for nonconcave/nonconvex uniformly elliptic equations with two space variables by the classical Nirenberg's result or assuming a Cordes assumption on the ellipticity, see the discussion in \cite{HuangAIHP,Gjmpa}.
\end{rem}
These Calder\'on-Zygmund estimates are the first step towards the proof of a fully nonlinear version of a regularity conjecture posed by P.-L. Lions for elliptic equations with superlinear gradient growth $\gamma>1$ \cite{Napoli}. It generally states that for strong solutions $u$ to 
\[
\Delta u\pm|Du|^\gamma=f(x)\text{ in }\Omega\subset\R^d,\text{ and }f\in L^q(\Omega)
\]
equipped with suitable (homogeneous) boundary conditions, one expects
\[
\|D^2u\|_{L^q(\Omega)}+\||Du|^\gamma\|_{L^q(\Omega)}\leq C(\|f\|_{L^q(\Omega)})
\]
provided that $q>\frac{n(\gamma-1)}{\gamma}$. In \cite{Napoli} it was discussed the validity of the estimate in the quadratic case via the Hopf-Cole transformation: this strategy applies only to equations with purely quadratic gradient terms and when $F=\Delta$. He also discussed the possible validity of a ``strong'' maximal regularity estimate where $C(\|f\|_q)=c\|f\|_q$ for some $c>0$, whose validity is in general still an open problem. These nonlinear estimates have been studied in two forms:
\begin{itemize}
\item Global, either in the case $\Omega$ is the flat torus and $f$ is a periodic forcing in $L^q$, or the problem is equipped with homogeneous Dirichlet or Neumann boundary conditions;
\item Local, if they appear as
\[
\||Du|^\gamma\|_{L^q(B_\frac12)}\leq C(\|f\|_{L^q(B_1)}).
\]
\end{itemize}
The conjecture was answered positively in \cite{CGell} for elliptic semilinear problems and in \cite{CGpar} for parabolic equations. Both these papers addressed the case of global estimates for periodic problems, with applications to Mean Field Games. Some works addressed the case of global estimates with boundary conditions: the paper \cite{GP} considered semilinear Neumann problems in convex domains, while \cite{CGL} treats quasilinear elliptic equations patterned over the $p$-Laplacian with Neumann boundary conditions, considering also the ``strong'' maximal regularity statement. The case of Dirichlet boundary conditions in the subquadratic case was the object of \cite{Gccm} and tackled via integral duality methods. Except some lower-order H\"older bounds \cite{CDLP}, the validity of maximal regularity for these equations has not been already investigated up to the boundary of the domain. see e.g. \cite[p. 4-5]{CV}. In these cases solutions are known to exhibit unnatural behaviors when the gradient has power-type growth, see e.g. the discussion in \cite{CV}. \\
As far as interior estimates are concerned, they have been addressed in \cite{CV} in the superquadratic regime $\gamma>2$, for nondivergence diffusions with $A\in W^{1,n}$, see also \cite{BardiPerthame} for earlier results when $\gamma=2$. A general interior estimate in the subquadratic and quadratic regime remains an open problem in the theory\footnote{Theorem 1.1 of the recent paper \cite{CirantWeietal} provided a unifying approach to obtain interior (nonlinear) Calder\'on-Zygmund estimates for semilinear PDEs driven by the Laplacian using blow-up methods and a delicate analysis of Morrey estimates in the whole regime $\gamma>\frac{n}{n-1}$, cf. Remark 2.2 therein.}. We give a further advance on this problem towards three directions. First, we provide a strategy to get maximal regularity estimates for problems with nonlinear diffusions and $H$ having natural growth in the gradient, but not necessarily purely quadratic. Our second aim is to weaken as much as possible the regularity requirements on the second order operator, so we consider a fully nonlinear second order operator, which is the prototype model of a nondivergence structure equation. Finally, we give a first answer concerning the validity of an interior estimate for gradient terms having natural and sub-natural growth, including the linear case $F=\Delta u$. The next summarizes the main result of the paper:
\begin{thm}\label{main}
Let $u\in W^{2,q}_{\mathrm{loc}}(\Omega)\cap L^q(\Omega)$ be a strong solution to \eqref{HJell}, and let $B_1\subset\Omega$, $q>q_E$. Assume that $F$ is uniformly elliptic and satisfies the assumptions of Lemma \ref{CZg}, with $H$ satisfying \eqref{assH}. Then we have the a priori estimate
\[
\|D^2u\|_{L^q(B_\frac12)}+\||Du|^2\|_{L^q(B_\frac12)}\leq C(\|u\|_{L^q(B_1)},\|f\|_{L^q(B_1)},q,n,\lambda,\Lambda,C_H,\beta_0).
\]
\end{thm}
\begin{ex}
Theorem \ref{main} applies, for instance, to the PDEs
\[
\inf_{\nu\in I}\{a_{ij}^\nu(x)\partial_{ij}u\}+H(x,Du)=f(x)\text{ or }\sup_{\nu\in I}\{a_{ij}^\nu(x)\partial_{ij}u\}+H(x,Du)=f(x),
\]
where $A^\nu(x)=\{a_{ij}^\nu(x)\}$ satisfies the uniform ellipticity condition and the existence of a constant $\nu_0$ such that for any ball $B_r(y)\subset B_1$
\[
\frac{1}{r^n}\int_{B_r(y)}|a_{ij}^\nu(x)-a_{ij}^\nu(y)|^n\,dx\leq \nu_0,
\]
$H$ has quadratic gradient growth and bounded coefficients, with $f\in L^q$, $q>q_E$.
\end{ex}
\begin{proof}
Let $u$ be a strong solution to \eqref{HJell}. Our goal is to give an estimate of $\||Du|^2\|_{L^q(B_{1/2})}$ in terms of the norms of the data on the greater ball $B_1$. The interpolation argument is similar to \cite{CV} for superqudratic gradient terms, though we treat here the quadratic case in a slightly different manner. We will denote by $C_i$ positive constants that may change from line to line. We first use the Gagliardo-Nirenberg-Miranda inequalities, cf. Theorem $1'$ in \cite{Nirenbergint}, to conclude
\begin{equation}\label{GNineq}
\|Du\|_{L^{2q}(B_{\sigma R})}\leq C_1\|D^2u\|_{L^z(B_{\sigma R})}^\theta\|u\|_{C^\alpha(B_{\sigma R})}^{1-\theta}+C_2\|u\|_{C^\alpha(B_{\sigma R})}\ ,
\end{equation}
for $\theta\in[\frac{1-\alpha}{2-\alpha},1)$, $q,z>1$ satisfying
\[
\frac{1}{2 q}=\frac1n+\theta\left(\frac1z-\frac2n\right)-(1-\theta)\frac{\alpha}{n},
\]
where $C_1,C_2$ do not depend on $R$. We apply the above inequality with $\theta=\frac{1-\alpha}{2-\alpha}$ satisfying $2\theta<1$ for any $\alpha\in(0,1)$. Note in particular that
\[
z=q\frac{2-2\alpha}{2-\alpha}<q,
\]
so for $\sigma R<1$ we have
\[
\|D^2u\|_{L^z(B_{\sigma R})}\leq \tilde{C}_1\|D^2u\|_{L^q(B_{\sigma R})}.
\]
Then, by regarding $u$ as a solution to the fully nonlinear Hessian equation
\[
F(x,D^2u)=H(x,u,Du)
\]
 we can apply \eqref{ineqGT} with $g=H$, use the growth conditions on $H$ and \eqref{GNineq} to conclude for $\sigma R<1$
\[
\|Du\|_{L^{2q}(B_{\sigma R})}\leq \frac{C_3}{(1-\sigma)^{2\theta}}\left(\|Du\|_{L^{2q}(B_{R})}^2+\|f\|_{L^q(B_R)}+\frac{\|u\|_{L^q(B_R)}}{R^2}\right)^\theta\|u\|_{C^\alpha(B_{\sigma R})}^{1-\theta}+C_2\|u\|_{C^\alpha(B_{\sigma R})}.
\]
Take $r\in(0,1)$, $R=\frac{1+r}{2}\in(\frac12,1)$ and $\sigma=\frac{r}{R}=\frac{2r}{1+r}$ (note that $\sigma\in(0,1)$). Then, $1-\sigma=\frac{1-r}{1+r}$. In view of the inequality \[(a+b)^\theta\leq a^\theta+b^\theta,\ \theta\in(0,1)\]
we conclude
\begin{equation}\label{eq1}
\|Du\|_{L^{2q}(B_r)}\leq C_3\frac{(1+r)^{2\theta}}{(1-r)^{2\theta}}\left(\|Du\|_{L^{2q}(B_{R})}^{2\theta}+\|f\|_{L^q(B_R)}^\theta+\frac{\|u\|^\theta_{L^q(B_R)}}{R^{2\theta}}\right)\|u\|_{C^\alpha(B_{\sigma R})}^{1-\theta}+C_2\|u\|_{C^\alpha(B_{\sigma R})}.
\end{equation}
We multiply \eqref{eq1} by $(1-r)^b$, $b>0$ to be chosen, and get
\begin{multline*}
(1-r)^b\|Du\|_{L^{2q}(B_r)}\leq C_3(1+r)^{2\theta}(1-r)^{b-2\theta}\left(\|Du\|_{L^{2q}(B_{R})}^{2\theta}+\|f\|_{L^q(B_R)}^\theta+\frac{\|u\|^\theta_{L^q(B_R)}}{R^{2\theta}}\right)\|u\|_{C^\alpha(B_{\sigma R})}^{1-\theta}\\
+C_2\|u\|_{C^\alpha(B_{\sigma R})}.
\end{multline*}
Then, we obtain
\begin{multline*}
(1-r)^b\|Du\|_{L^{2q}(B_r)}\leq C_3\left[\frac{2^{2\theta}(1-r)^{b-2\theta}}{\left(1-\frac{1+r}{2}\right)^{2b\theta}}\left(1-\frac{1+r}{2}\right)^{2b\theta}\|Du\|_{L^{2q}(B_{R})}^{2\theta}+\|f\|_{L^q(B_R)}^\theta\right.\\
\left.+\frac{\|u\|^\theta_{L^q(B_R)}}{R^{2\theta}}\right]\|u\|_{C^\alpha(B_{\sigma R})}^{1-\theta}+C_2\|u\|_{C^\alpha(B_{\sigma R})}\\
=C_3\left\{2^{2\theta+2b\theta}(1-r)^{b-2\theta-2b\theta}\left[\left(1-R\right)^{b}\|Du\|_{L^{2q}(B_{R})}\right]^{2\theta}\right.\\
\left.+\|f\|_{L^q(B_R)}^\theta+\frac{\|u\|^\theta_{L^q(B_R)}}{R^{2\theta}}\right\}\|u\|_{C^\alpha(B_{\sigma R})}^{1-\theta}+C_2\|u\|_{C^\alpha(B_{\sigma R})}.
\end{multline*}

We set $\Phi(r):=(1-r)^b\|Du\|_{L^{2q}(B_r)}$, $r\in (0,1)$, and set $b=\frac{2\theta}{1-2\theta}>2\theta$ to get 
\[
\Phi(r)\leq C_4\left[\left(1-R\right)^{b}\|Du\|_{L^{2q}(B_{R})}\right]^{2\theta}\|u\|_{C^\alpha(B_{\sigma R})}^{1-\theta}+C_5= C_4(\Phi(R))^{2\theta}\|u\|_{C^\alpha(B_{\sigma R})}^{1-\theta}+C_5,
\]
where
\[
C_5:=C_3\|f\|_{L^q(B_R)}^\theta\|u\|_{C^\alpha(B_{\sigma R})}^{1-\theta}+C_3\frac{\|u\|^\theta_{L^q(B_R)}}{R^{2\theta}}\|u\|_{C^\alpha(B_{\sigma R})}^{1-\theta}+C_2\|u\|_{C^\alpha(B_{\sigma R})}.
\]
We apply the generalized Young's inequality (being $2\theta<1$) to conclude that
\[
\Phi(r)\leq \frac12\Phi(R)+\frac{1}{(R-r)^{2\theta}}\underbrace{C_3\|u\|^\theta_{L^q(B_R)}\|u\|_{C^\alpha(B_{\sigma R})}^{1-\theta}}_{C_7}+\underbrace{C_6(\|u\|_{C^\alpha(B_{\sigma R})}^\frac{1-\theta}{1-2\theta}+\|u\|_{C^\alpha(B_{\sigma R})})}_{C_8},
\]
from which we deduce by Lemma 4.3 in \cite{HanLin} (which applies to inequalities of the form $\Phi(r)\leq \mu \Phi(R)+\frac{K_1}{(R-r)^\eta}+K_2$, $0<r<R<1$, $0<\mu<1$, $\eta>0$) the bound
\[
\Phi(r)\leq C_9\left(\frac{C_7}{(R-r)^{2\theta}}+C_8\right),\ r\in(0,1).
\]
In view of the above choice of $\theta=\frac{1-\alpha}{2-\alpha}$ satisfying $2\theta<1$ for any $\alpha\in(0,1)$, we use the bound on $\|u\|_{C^\alpha(B_{\sigma R})}$ from Corollary \ref{Holder} valid for $q>q_E$ (note that it can be applied since strong solutions are $L^p$ viscosity solutions, see Proposition 2-(i) in \cite{SwiechDCDS}). We then take $r=1/2$ (hence $R=\frac34$ and $\sigma=\frac12$) and get the estimate
\[
\|D^2u\|_{L^{q}(B_\frac12)}+\||Du|^2\|_{L^{q}(B_\frac12)}\leq C(\|u\|_{L^q(B_1)},\|f\|_{L^q(B_1)},q,n,\lambda,\Lambda,C_H,\beta_0).
\]
\end{proof}

\begin{rem}
Regularity estimates in $C^{1,\alpha}$ and $W^{2,q}$ starting from $L^p$ viscosity solutions of more general equations with quadratic growth and $p>n$ can be found in Theorem 5.1 of \cite{SwiechDCDS}. Theorem \ref{main} is stated as an a priori estimate, and it does not give a regularity estimate as in \cite{SwiechDCDS}, which takes into account even unbounded coefficients. 
\end{rem}
\begin{rem}
Some comments concerning the existence of strong solutions are in order. Though some existence results for equations with quadratic or subquadratic gradient terms are available (with possibly unbounded coefficients), an explicit strong solvability result for the equation \eqref{HJell} is not available in the literature. However, we believe that the scheme of \cite{KoikeSwiechJapan} could be a starting point for proving an existence result via the interior bound in Theorem \ref{main}. This would require the validity of a H\"older bound as in Corollary \ref{Holder} up to the boundary under suitable geometric conditions on the domain. These were shown to hold in the case of linear gradient growth in Theorem 6.2 of \cite{KoikeSwiechJapan} and in Theorem 2 of \cite{Sirakov} for the quadratic growth with $f\in L^n$. We refer also to \cite{CDLP} for H\"older bounds of subsolutions to equations with superquadratic growth and $f\in L^\infty$. This topic will be the matter of a future research.  
\end{rem}

\begin{rem}
The subquadratic case follows similarly by interpolation. In this case one only needs to interpolate $\||Du|^\gamma\|_{L^q}$ with the $L^\infty$ norm of $u$. More precisely, we have
\[
\|Du\|_{L^{\gamma q}(B_{\sigma R})}^\gamma\leq C\|Du\|_{L^{2q}(B_{\sigma R})}^\gamma\leq C_1\|D^2u\|_{L^q(B_{\sigma R})}^\frac{\gamma}{2}\|u\|_{L^\infty(B_{\sigma R})}^{\frac{\gamma}{2}}+C_2\|u\|_{L^\infty(B_{\sigma R})}^\gamma\ .
\]
One can then exploit the weighted Young's inequality since $\gamma<2$ and conclude the maximal regularity estimate. Note that in the case of fully nonlinear diffusions
\[
q>q_E\geq \frac{n}{2}>\frac{n(\gamma-1)}{\gamma},\ 1<\gamma<2.
\]
Lower order estimates, even at the level of $C^\alpha$ spaces, in this setting can be found in \cite{KoikeADE}. Our technique fails in the superquadratic regime, as an explicit H\"older exponent is needed to run the argument. We also note that when $\gamma>2$ one could expect maximal regularity when
\[
q>\max\left\{q_E,\frac{n(\gamma-1)}{\gamma}\right\},
\]
but this remains open for general uniformly elliptic operators $F$.
\end{rem}
\begin{rem}
When $F=\Delta$ we have $q_E=n/2$, cf. \cite{Escauriaza}. Any $u\in W^{2,q}$ is also an energy solution, and one can apply the H\"older regularity of Corollary 4.23 in \cite{HanLin} that is based on a similar strategy, but exploits the De Giorgi-Nash-Moser theory, see Remark \ref{div}. This result, combined with the approach of Theorem \ref{main}, gives a new interior maximal regularity result for the viscous problem
\[
-\Delta u+H(Du)=f(x)\in L^q.
\]
\end{rem}
\begin{rem}
One can prove more general maximal regularity estimates and H\"older regularity results for equations having natural growth in the gradient and unbounded coefficients in $L^s$ spaces. One can also consider more general lower-order terms with linear gradient growth. Weak Harnack inequalities as well as local maximum principles can be found in \cite{KSjfpta,KoikeSwiechJapan,KSjmpa} and references therein.
\end{rem}

The extension to the parabolic case is straightforward. One can use the weak Harnack inequality and the local maximum principle from \cite{Wang1} along with the analogue of \eqref{ineqGT} for parabolic equations, see e.g. p.170 in \cite{Lieberman}. Note that the latter can be extended to the fully nonlinear setting as in Proposition 3.5 of \cite{KS}. We conclude with the statement, without proof, of the maximal regularity result for the parabolic equation \eqref{HJpar} (here $W^{2,1}_q=\{\partial_t u,u,Du,D^2u\in L^q\}$ with $q>q_E\in[\frac{n+2}{2},n+1)$). This also provides the first interior maximal $L^q$-regularity estimate for equations driven by the heat operator and quadratic gradient growth: results in this direction for the subquadratic and superquadratic growth case can be found in \cite{CGpar,C} by integral methods.
\begin{thm}
Let $u\in W^{2,1}_{q,\mathrm{loc}}(Q_1)\cap L^q(Q_1)$ be a strong solution to \eqref{HJpar}. Assume that $F$ is uniformly parabolic and that the equation $F(x,t,D^2u)-\partial_t u=g(x,t)$ admits parabolic interior $W^{2,1}_q$ estimates. Assume also that $H$ satisfies \eqref{assHpar}. Then we have the a priori estimate
\[
\|\partial_t u\|_{L^q(Q_\frac12)}+\|D^2u\|_{L^q(Q_\frac12)}+\||Du|^2\|_{L^q(Q_\frac12)}\leq C(\|u\|_{L^q(Q_1)},\|f\|_{L^q(Q_1)},q,n,\lambda,\Lambda,C_H,\beta_0).
\]
\end{thm}

\providecommand{\bysame}{\leavevmode\hbox to3em{\hrulefill}\thinspace}
\providecommand{\MR}{\relax\ifhmode\unskip\space\fi MR }

\providecommand{\MRhref}[2]{
  \href{http://www.ams.org/mathscinet-getitem?mr=#1}{#2}
}
\providecommand{\href}[2]{#2}

\end{document}